\documentclass[a4paper,12pt]{article}

\usepackage{amsmath,amsfonts,amssymb,amsthm,amscd}
\usepackage {amsfonts, amssymb, amscd,amsthm, amsmath, eucal}

\usepackage{amsbsy}
\usepackage[all]{xy}

\usepackage{color}

\theoremstyle{plain} {
  \newtheorem{thm}{Theorem}[section]
  \newtheorem{defn}[thm]{Definition}
  \newtheorem{cor}[thm]{Corollary}
  \newtheorem{lem}[thm]{Lemma}

  \theoremstyle{definition}
  \newtheorem{rem}[thm]{Remark}
    
  \theoremstyle{plain}

  \newtheorem{sssection}[thm]{}
}

{
}
\renewcommand{\subsubsection}{\sssection\rm}

\newcommand{\bE}{\mathbf E}
\newcommand{\bG}{\mathbf G}
\newcommand{\bH}{\mathbf H}

\newcommand{\bT}{\mathbf T}

\newcommand{\bC}{\mathbf C}
\newcommand{\bM}{\mathbf M}

\newcommand{\cH}{\mathcal H}

\newcommand{\cO}{\mathcal O}

\DeclareMathOperator{\Iso}{Iso}
\DeclareMathOperator{\spec}{Spec}

\newcommand{\can}{\text{\rm can}}

\newcommand{\pr}{\text{\rm pr}}

\newcommand{\Spec}{\text{\rm Spec}}

\newcommand{\Aut}{\text{\rm Aut}}

\newcommand{\Aff}{\mathbf {A}}

\newcommand \xra {\xrightarrow }

\newcommand \hra {\hookrightarrow }

\renewcommand \phi\varphi

\begin{document}

\title{A short exact sequence
}

\author{Ivan Panin\footnote{The author acknowledges support of the
RFBR grant No. 19-01-00513.}
}

\maketitle

\begin{abstract}
Let $R$ be a regular semi-local integral domain containing a field and $K$ be its fraction field.
Let $\mu: \bG \to \bT$
be an $R$-group schemes morphism between reductive $R$-group schemes, which is smooth as a scheme morphism.
Suppose that $T$ is an $R$-torus.
Then the map
$\bT(R)/\mu(\bG(R)) \to \bT(K)/\mu(\bG(K))$
is injective and certain purity theorem is true.
These and other results are derived from
an extended form of Grothendieck--Serre conjecture proven in the present paper for rings
$R$ as above.
\end{abstract}

\section{Main results}\label{Introduction}
Let $R$ be a commutative unital ring. Recall that an $R$-group scheme $\bG$ is called reductive, (respectively, semi-simple or simple),
if it is affine and smooth as an $R$-scheme and if, moreover,
for each algebraically closed field $\Omega$ and for each ring homomorphism $R\to\Omega$ the scalar extension $\bG_\Omega$ is
a connected reductive (respectively, semi-simple or simple)
algebraic group over $\Omega$.
The class of reductive group schemes contains
the class of semi-simple group schemes which in turn contains the class of simple group schemes.
This notion of a reductive $R$-group scheme
coincides with~\cite[Exp.~XIX, Definition~2.7]{SGA3}.
This notion of a simple $R$-group scheme coincides with the notion of a simple semi-simple $R$-group
scheme from Demazure and Grothendieck \cite[Exp. XIX, Definition 2.7 and Exp. XXIV, 5.3]{SGA3}.
Here is our first main result based on results of \cite{Pan2} and \cite{Pan3} and significantly extending the corresponding results of
\cite{Pan2} and \cite{Pan3}.
\begin{thm}
\label{Aus_Buksbaum_3}
Let $R$ be a regular semi-local integral domain containing a field. Let $K$ be the fraction field of $R$.
Let
$\mu: \bG \to \bT$
be an $R$-group scheme morphism between reductive $R$-group schemes, which is smooth as a scheme morphism.
Suppose $T$ is an $R$-torus.
Then the map
$\bT(R)/\mu(\bG(R)) \to \bT(K)/\mu(\bG(K))$
is injective and the sequence
\begin{equation}
\label{Aus_Buks_sequence_3}
\{1\} \to \bT(R)/\mu(\bG(R)) \to
\bT(K)/\mu(\bG(K)) \xrightarrow{\sum r_{\mathfrak p}} \bigoplus_{\mathfrak p}
\bT(K)/[\bT(R_{\mathfrak p})\cdot \mu(\bG(K))] \to \{1\}
\end{equation}
is exact, where $\mathfrak p$ runs over all height one prime ideals of $R$
and each
$r_{\mathfrak p}$ is the natural map (the projection to the factor group).
\end{thm}
Let us comment on the first assertion of the theorem. Let $\bH$ be the kernel of $\mu$. It turns out that $\bH$ is
a quasi-reductive $R$-group scheme (see Definition \ref{pre-reductive}).
There is a sequence of group sheaves
$1\to \bH\to \bG\to \bT\to 1$,
which is exact in the \'{e}tale topology on $Spec R$.
Theorem \ref{MainThm1} yields now the injectivity of the map
$\bT(R)/\mu(\bG(R)) \to \bT(K)/\mu(\bG(K))$.

\begin{thm}\label{s_simple_groups}
Let $R$ be a regular semi-local integral domain containing a field. Let $K$ be the fraction field of $R$.
Let $\bG_1$ and $\bG_2$ be two semi-simple $R$-group schemes.
Suppose the generic fibres $\bG_{1,K}$ and $\bG_{2,K}$ are isomorphic as algebraic $K$-groups. Then
the $R$-group schemes $\bG_1$ and $\bG_2$ are isomorphic.
\end{thm}
%This theorem can not be directly derived from \cite{Ni} and \cite{NG}. Indeed,
%only geometrically connected group schemes are regarded there.
To prove Theorem \ref{s_simple_groups}
we need to work with automorphism group scheme of a semi-simple $R$-group scheme. The latter group scheme
is not geometrically connected in general.
So, Theorem \ref{s_simple_groups} can not be derived from \cite{FP} and \cite{Pan3}.

We state right below a theorem, which asserts that an extended version of Grothendieck--Serre conjecture holds for rings
$R$ as above. This latter theorem is proved in this paper. Theorem \ref{s_simple_groups} and the first assertion of
Theorem \ref{Aus_Buks_sequence_3} are derived from it.
To state the mentioned theorem it is convenient to give the following.
\begin{defn}[quasi-reductive]\label{pre-reductive}
Assume that $S$ is a Noetherian commutative ring. An $S$-group scheme $\bH$ is called {\rm quasi-reductive}
if there is a finite \'{e}tale $S$-group scheme $\bC$ and a smooth $S$-group scheme morphism
$\lambda: \bH \to \bC$ such that its kernel is a reductive $S$-group scheme and $\lambda$ is surjecive
locally in the \'{e}tale topology on $S$.
\end{defn}
Clearly, reductive $S$-group schemes are quasi-reductive.
Quasi-reductive $S$-group schemes are {\it affine and smooth} as $S$-schemes.
There are two types of quasi-reductive $S$-group schemes,
which we are focusing on in the present paper. The first one is
the automorphism group scheme of a semi-simple $S$-group scheme. The second one is obtained as follows:
take a reductive $S$-group scheme $\bG$, an $S$-torus $\bT$ and a smooth $S$-group morphism
$\mu: \bG \to \bT$. Then one can check that the kernel $\bH$ of $\mu$ is quasi-reductive.
It is an extension of a finite \'{e}tale $S$-group scheme $\bC$ of multiplicative type via a reductive
$S$-group scheme $\bG_0$.

Assume that $U$ is a regular scheme, $\bH$ is a quasi-reductive $U$-group scheme.
Recall that a $U$-scheme $\cH$ with an action of $\bH$ is called \emph{a principal $\bH$-bundle over $U$},
if $\cH$ is faithfully flat and quasi-compact over $U$ and the action is simple transitive,
that is, the natural morphism $\bH\times_U\cH\to\cH\times_U\cH$ is an isomorphism, see~\cite[Section~6]{Gr4}.
Since $\bH$ is $S$-smooth, such a bundle is trivial locally in \'{e}tale topology but in general not in Zariski topology.
Grothendieck and Serre conjectured that for a reductive $U$-group scheme $\bH$ a principal $\bH$-bundle $\cH$
over $U$  is trivial locally in Zariski topology, if it is trivial generically.
A {\it survey paper} on the topic is \cite{P2}.

The conjecture is true, if $\Gamma(U,\cO_U)$ contains a field (see \cite{FP} and \cite{Pan3}).
It is proved in \cite{Ni} that the conjecture is true in general for discrete valuation rings.
This result is extended in \cite{PSt} to the case of semi-local Dedekind integral domains
assuming that $\bG$ is simple simply connected and isotropic in a certain precise sense.
In \cite{NG} results of \cite{Ni} and \cite{PSt} are extended further.
It is proved there that the conjecture is true in general for the case of semi-local Dedekind integral domains.
The following result is a further extension of the main theorem of \cite{Pan3}.
\begin{thm}\label{MainThm1}
Let $R$ be a regular semi-local integral domain containing a field. Let $K$ be the fraction field of $R$.
Let $\bH$ be a quasi-reductive group scheme over $R$. Then the map
\[
  \textrm{H}^1_{\text{\'et}}(R,\bH)\to \textrm{H}^1_{\text{\'et}}(K,\bH),
\]
\noindent
induced by the inclusion of $R$ into $K$ has a trivial kernel.
In other words, under the above assumptions on $R$ and $\bH$, each principal $\bH$-bundle over $R$ having a $K$-rational point is trivial.
\end{thm}

\begin{cor}\label{CorOfMainThm1}
Under the hypothesis of Theorem~\ref{MainThm1}, the map
\[
  \textrm{H}^1_{\text{\'et}}(R,\bH)\to \textrm{H}^1_{\text{\'et}}(K,\bH),
\]
\noindent
induced by the inclusion of $R$ into $K$, is injective. Equivalently, if $\cH_1$ and $\cH_2$ are two principal $\bH$-bundles isomorphic over
$\Spec K$,
then they are isomorphic.
\end{cor}

\begin{proof}
Let $\cH_1$ and $\cH_2$ be two principal $\bH$-bundles isomorphic over
$\spec K$.
Let $\Iso(\cH_1,\cH_2)$ be the scheme of isomorphisms of principal $\bH$-bundles.
This scheme is a principal $\Aut\cH_1$-bundle.
By Theorem~\ref{MainThm1} it is trivial, and we see that $\cH_1\cong\cH_2$.
\end{proof}

Theorems  \ref{MainThm1} and \ref{s_simple_groups} are proved in Section \ref{2_theorems}.
Theorem \ref{Aus_Buksbaum_3} is proved in Section \ref{SectAus_Buksbaum_3}.

\section{Proof of Theorems \ref{MainThm1} and \ref{s_simple_groups}}\label{2_theorems}
We begin with the following general
\begin{lem}\label{f_flat_morphism}
Let $X$ be a regular semi-local integral domain. Let $\pi: X'\to X$ be a finite morphism.
Let $\eta\in X$ be the generic point of $X$. Then
sections of $\pi$
over $X$ are in the bijection with sections of $\pi$ over $\eta$.
\end{lem}

\begin{proof}
It suffices to check that each section $s: \eta\to X'$ of $\pi$
can be extended to a section of $\pi$ over $X$.

Decompose $\pi$ as a composition
$X' \xra{i} \Aff^n_X \xra{p} X$, where $p$ is the projection and $i$ is a closed embedding.
%Clearly, $\Pro^n(X)=\Pro^n(\eta)$. Let $\Pro^n_X$ be the projective space containing $\Aff^n_X$
%as an open subscheme. Let $\bar p: \Pro^n_X \to X$ be the projection.
%Since $\pi$ is finite $i(X')$ is closed in $\Pro^n_X$.
Let $s: \eta\to X'$ be a section of $\pi$. Since $X$ is a regular semi-local and $\pi$ is projective
there is a closed codimension two
subset $Z$ in $X$ and a section $\phi: X-Z \to X'$ of $\pi$ that extends $s$.
Since $\Gamma(X, \mathcal O_X)=\Gamma(X-Z, \mathcal O_X)$ the section $\phi$
is extended to a section $\tilde \phi$ of $\pi$. The lemma is proved.
\end{proof}

\begin{cor}\label{f_gr_scheme}
Let $X$, $\eta\in X$ be as in the previous lemma and $\bE$ be a finite \'{e}tale group $X$-scheme.
Then
the $\eta$-points of $\bE$ coincides with the $X$-points of $\bE$.
\end{cor}

\begin{cor}\label{inj_on_H1}
Under the hypothesis of Corollary \ref{f_gr_scheme}
the kernel of the pointed set map
$H^1_{\text{\'et}}(X,\bE)\to H^1_{\text{\'et}}(\eta,\bE)$
is trivial.
\end{cor}

\begin{proof}
Let $\cal E$ be a principal $\bE$-bundle over $X$. The standard descent arguments shows that
the $X$-scheme $\cal E$ is finite and \'{e}tale. Thus, ${\cal E}(X)=\cal E(\eta)$. This proves the corollary.
\end{proof}

\begin{proof}[Proof of Theorem \ref{MainThm1}]
Since $\bH$ is quasi-reductive $R$-group scheme, there is a finite \'{e}tale $R$-group scheme $\bC$ and a smooth $R$-group scheme morphism
$\lambda: \bH \to \bC$ such that its kernel $\bG$ is a reductive $R$-group scheme and $\lambda$ is
surjecive locally in the \'{e}tale topology on $S$.
The sequence of the \'{e}tale sheaves
$1\to \bG\to \bH\to \bC\to 1$ is exact. Thus, it induces a commutative diagram of pointed set maps with exact rows
$$
\SelectTips{cm}{}
\xymatrix @C=2pc @R=10pt {
\bC(R) \ar[r]^{\partial} \ar[dd]_{\alpha} & {H_\text{{\'et}}^1(R,\bG)} \ar[r] \ar[dd]_{\beta} & {H_\text{{\'et}}^1(R,\bH)} \ar[r] \ar[dd]_{\gamma} & {H_\text{{\'et}}^1(R,\bC)} \ar[dd]_{\delta} \\
{} & {} & {} & {} & {} \\
\bC(K) \ar[r]^{\partial} & {H_\text{{\'et}}^1(K,\bG)} \ar[r] & {H_\text{{\'et}}^1(K,\bH)} \ar[r] & {H_\text{{\'et}}^1(K,\bC)} \\
}
$$
The map $\alpha$ is bijective by Corollary \ref{f_gr_scheme}, the map $\delta$ has the trivial kernel by Corollary
\ref{inj_on_H1}, the map $\beta$ is injective by
\cite[Corollary 1.2]{Pan3}
%\ref{CorOfMainThm1}.
Now a simple diagram chase shows that $ker(\gamma)=\ast$. This proves the theorem.
\end{proof}

\begin{rem}\label{Ch_G_P}
The statement of \cite[Lemma]{Ch-G-P} and its proof are non-acurate both, since
the authors do not assume
the injectivity of the map
$H_\text{{\'et}}^1(R,\bG^0)\to H_\text{{\'et}}^1(K,\bG^0_R)$.
\end{rem}

\begin{proof}[Proof of Theorem \ref{s_simple_groups}]
The $R$-group scheme $\underline {\text {Aut}}:=\underline {\text {Aut}}_{R-gr-sch}(\bG_1)$
is quasi-reductive by
\cite{D-G}.
The $R$-scheme
$\underline {\text {Iso}}:=\underline {\text {Iso}}_{R-gr-sch}({\bf G_1},{\bf G_2})$
is a principal $\underline {\text {Aut}}$-bundle.
An isomorphism $\phi: \bG_{1,K} \to \bG_{2,K}$ of algebraic $K$-groups gives a section of
$\underline {\text {Iso}}$ over $K$. So, $\underline {\text {Iso}}_K$
is a trivial principal $\underline {\text {Aut}}_K$-bundle.
Hence $\underline {\text {Iso}}$ is a trivial
principal $\underline {\text {Aut}}$-bundle by Theorem \ref{MainThm1}.
Thus, it has a section over $R$. So, there is an $R$-group scheme isomorphism
$\bG_1 \cong \bG_2$.
\end{proof}

\section{Proof of the first assertion of Theorem \ref{Aus_Buksbaum_3}}
\label{One_Lemma}
\begin{lem}\label{Kernel_G_to_T}
Let $X$ be a regular irreducible affine scheme.
Let $\bG$ be a reductive $X$-group scheme and $\bT$ be an $X$-torus.
Let $\mu: \bG \to \bT$
be an $X$-group schemes morphism, which is smooth as a scheme morphism.
Then the kernel of $\mu$ is
a quasi-reductive $X$-group scheme.
\end{lem}

\begin{proof}
Consider the coradical $Corad(\bG)$ of $\bG$ together with the canonical
$X$-group morphism $\alpha: \bG\to Corad(\bG)$. By the universal property of the $X$-group morphism $can$
there is a unique $X$-group morphism $\bar \mu: Corad(\bG)\to T$ such that
$\mu=\bar \mu\circ \alpha$. Since $\mu$ is surjective locally for the \'{e}tale topology, hence so is $\bar \mu$.
Let $\text{ker}(\bar \mu)$ be the kernel of $\bar \mu$ and let $\bH:=\alpha^{-1}(\text{ker}(\bar \mu))$
be the scheme theoretic pre-image of $\text{ker}(\bar \mu)$. Clearly, $\bH$ is a closed $X$-subgroup scheme of $\bG$, which
is the kernel of $\mu$. We must check that $\bH$ is a quasi-reductive.

The $X$-group scheme $\text{ker}(\bar \mu)$ is of multiplicative type. Hence there is a
finite $X$-group scheme $\bM$ of  multiplicative type and a faithfully flat $X$-group scheme morphism
$can: \text{ker}(\bar \mu)\to \bM$, which has the following property:
for any finite $X$-group scheme $\bM'$ of  multiplicative type
and an $X$-group morphism $\phi: \text{ker}(\bar \mu)\to \bM'$ there is a unique $X$-group morphism
$\psi: \bM\to \bM'$ with $\psi \circ can=\phi$. It is known that the kernel of $can$ is an $X$-torus.
Call it $\bT^0$. Since $\mu$ is smooth, hence so is $\bar \mu$. Thus, the $X$-group scheme
$\text{ker}(\bar \mu)$ is an $X$-smooth scheme. This yields that $M$ is \'{e}tale over $X$.

Let $\beta=\alpha|_{\bH}: \bH\to \text{ker}(\bar \mu)$
and let $\bG^0:=\beta^{-1}(\bT^0)$
be the scheme theoretic pre-image of $\bT^0$.
Clearly, $\bG^0$ is a closed $X$-subgroup scheme of $\bH$, which
is the kernel of the morphism $can\circ \beta: \bH\to \bM$.
Let $\gamma=\beta|_{\bG^0}: \bG^0\to \bT^0$.

The $X$-group scheme $\bM$ is finite and \'{e}tale.
The morphism $can$ is smooth. The morphism $\beta$ is smooth as
a base change of the smooth morphism $\alpha$. Thus,
$\lambda:=can\circ \beta$ is smooth.
It is also surjective locally in the \'{e}tale topology on $X$, because $can$ and $\beta$ have this property.
By the construction $\bG^0=\text{ker}(\lambda)$.
So, to prove that $\bH$ is quasi-reductive
it remains to check the reductivity of $\bG^0$.

The $X$-group scheme $\bG^0$ is affine as a closed $X$-subgroup scheme of the reductive $X$-group scheme $\bG$.
Prove now that $\bG^0$ is smooth over $X$. Indeed,
the morphism $\gamma$ is smooth as a base change of the smooth morphism $\alpha$.
The $X$-scheme $\bT^0$ is smooth, since it is an $X$-torus.
Thus, the $X$-scheme $\bG^0$ is smooth.

Write $X$ as $Spec S$ for a regular integral domain $S$.
It remains to verify that for
each algebraically closed field $\Omega$ and for each ring homomorphism $S\to\Omega$ the scalar extension $\bG^0_\Omega$ is
a connected reductive
algebraic group over $\Omega$.
Firstly, recall that $\text{ker}(\alpha)$ is a semi-simple $S$-group scheme. It is the
$S$-group scheme $\bG^{ss}$ under the notation of \cite{D-G}. Clearly,
$\text{ker}(\gamma)=\text{ker}(\alpha)$. Thus, $\text{ker}(\gamma)=\bG^{ss}$
is a semi-simple $S$-group scheme. Since the morphism $\gamma$ is smooth for
each algebraically closed field $\Omega$ and for each ring homomorphism $S\to\Omega$
we have an exact sequence of smooth algebraic groups over $\Omega$
$$1\to \bG^{ss}_\Omega \to \bG^0_\Omega \to \bT^0_\Omega \to 1.$$
The groups $\bT^0_\Omega$, $\bG^{ss}_\Omega$ are connected. Hence the group $\bG^0_\Omega$
is connected too. We know already that it is affine.

Finally, check that its unipotent radical $\mathbf U$ of $\bG^0_\Omega$ is trivial.
Since there is no non-trivial $\Omega$-group morphisms
$\mathbf U\to \bT^0_\Omega$, we conclude that $\mathbf U\subset \bG^{ss}_\Omega$. Since $\bG^{ss}_\Omega$
is semi-simple one has $\mathbf U=\{1\}$. This completes the proof of the reductivity of the $R$-group scheme $\bG^0$.
Thus, the $R$-group scheme $\bH$ is quasi-reductive. This proves the lemma.
\end{proof}

\begin{proof}[Proof of the first assertion of Theorem \ref{Aus_Buksbaum_3}]
Let $\bH$ be the kernel of $\mu$.
Since $\mu$ is smooth, the group scheme sequence
$$1 \to \bH \to \bG \to \bT \to 1$$
gives rise to an short exact sequence of group sheaves in the \'{e}tale topology.
In turn that sequence of sheaves induces a long exact sequence of pointed sets.
So, the boundary map
$\partial: \bT(R) \to \textrm{H}^1_{\text{\'et}}(R,\bH)$
fits in a commutative diagram
$$
\CD
\bT(R)/\mu (\bG(R)) @>>> \textrm{H}^1_{\text{\'et}}(R,\bH) \\
@VVV @VVV \\
\bT(K)/\mu (\bG(K)) @>>> \textrm{H}^1_{\text{\'et}}(K,\bH). \\
\endCD
$$
Clearly, the horizontal arrows have trivial kernels.
The right vertical arrow has trivial kernel by Lemma \ref{Kernel_G_to_T} and Theorem \ref{MainThm1}.
Thus the left vertical arrow has trivial kernel too. Since it is a group homomorphism,
it is injective.
\end{proof}

\section{Norms}
\label{SectNorms}
We recall here a construction from \cite{P1}.
%In sections \ref{SectNorms}, \ref{SectUnramifiedElements},
%\ref{SectSpecializationMaps}
%we prove results which will be used to prove Theorems
%the rest of the paper we prove few results which we refer to reducing Theorems
%\ref{TheoremA}, \ref{TheoremA1}
%and
%\ref{MainThmGeometric}.
%to Theorem \ref{equating3}.
Let $k\subset K\subset L$ be field extensions and assume that $L$
is finite separable over $K$. Let $K^{sep}$ be a separable closure
of $K$ and $\sigma_i:K\to K^{sep},\enspace 1\leq i\leq n$
the different embeddings of $K$ into $L$. Let $C$ be a $k$-smooth
commutative algebraic group scheme defined over $k$. One can define
a norm map
$${\mathcal N}_{L/K}:C(L)\to C(K)$$
by
${\mathcal N}_{L/K}(\alpha)=\prod_i C(\sigma_i)(\alpha) \in C(K^{sep})^{{\mathcal G}(K)} =C(K)\ .
$
In \cite{P1} following Suslin and Voevodsky
\cite[Sect.6]{SV} we
generalize this construction to finite flat ring extensions.
\smallskip
Let $p:X\to Y$ be a finite flat morphism of affine schemes.
Suppose that its rank is constant, equal to $d$. Denote by
$S^d(X/Y)$ the $d$-th symmetric power of $X$ over $Y$.

\smallskip
Let $k$ be a field. Let $\mathcal O$ be the semi-local ring of finitely many {\bf closed} points
on a smooth affine irreducible $k$-variety $X$.
Let $C$ be an affine smooth commutative $\mathcal O$-group scheme,
Let $p:X\to Y$ be a finite flat morphism of affine $\mathcal O$-schemes
and $f:X\to C$ any $\mathcal O$-morphism.
In \cite{P1} {\it the norm} $N_{X/Y}(f)$ of $f$ {\it is defined as the composite map}
\begin{equation}
\label{NormMap}
Y \xra{N_{X/Y}}
S^d(X/Y) \to S^d_{\mathcal O}(X) \xra{S^d_{\mathcal O}(f)} S^d_{\mathcal O}(C)\xra{\times} C
\end{equation}
Here we write $"\times"$ for the group law on $C$.
The norm maps $N_{X/Y}$ satisfy the following conditions
\begin{itemize}
\item[(i')]
Base change: for any map $f:Y'\to Y$ of affine schemes, putting
$X'=X\times_Y Y'$ we have a commutative diagram
$$
\begin{CD}
C(X)                  @>{(id \times f)^{*}}>>            C(X^{\prime})      \\
@V{N_{X/Y}}VV @VV{N_{X^{\prime}/Y^{\prime}}}V    \\
C(Y)                  @>{f^{*}}>>            C(Y^{\prime})
\end{CD}
$$
\item[(ii')]
multiplicativity: if $X=X_1 \amalg X_2$ then the diagram commutes
$$
\begin{CD}
C(X)                  @>{(id \times f)^{*}}>>            C(X_1) \times C(X_2)      \\
@V{N_{X/Y}}VV                              @VV{N_{X_1/Y}N_{X_2/Y}}V    \\
C(Y)                  @>{id}>>            C(Y)
\end{CD}
$$
\item[(iii')]
normalization: if $X=Y$ and the map $X \to Y$ is the identity then $N_{X/Y}=id_{C(X)}$.
\end{itemize}

\section{Unramified elements}
\label{SectUnramifiedElements}
Let $k$ be a field, $\mathcal O$ be the semi-local ring of finitely many closed points
on a $k$-smooth irreducible affine $k$-variety $X$.
%$\mathcal O$ be the $k$-algebra from Theorem
%\ref{Aus_Buksbaum}
Let $K$ be the fraction field of $\mathcal O$, that is $K=k(X)$.
Let
$$\mu: \bG \to \bT$$
be a smooth $\mathcal O$-morphism of reductive
$\mathcal O$-group schemes, with a torus $\bT$.
%Set $\bH= \ker (\mu)$.
%By Lemma \ref{One_Lemma}  it is
%a quasi-reductive $\mathcal O$-group scheme.
%Thus, it satisfies the conclusion of \ref{MainThm1}
%for each regular semi-local integral domain $R$ which is an $\mathcal O$-algebra.
%As an obvious consequence, the group homomorphism
%$$\bC(R)/\mu(\bG(R)) \to \bC(K)/\mu(\bG(K)),$$
%is injective, where $K$ is the fraction field of $R$.
%Set
%$H= ker (\mu)$ and suppose additionally that
%$H$
%Suppose additionally that
%the kernel of $\mu$ is a reductive $\mathcal O$-group scheme.
%Let
%$\mu: G\to C$
%be the morphism of reductive $\mathcal O$-group schemes from Theorem
%\ref{Aus_Buksbaum}.
We work in this section with {\it the category of commutative Noetherian $\mathcal O$-algebras}.
For a commutative $\mathcal O$-algebra $S$ set
\begin{equation}
\label{MainFunctor}
{\cal F}(S)=\bT(S)/\mu(\bG(S)).
\end{equation}
For an element $\alpha \in \bT(S)$ we will write $\bar \alpha$ for its
image in ${\cal F}(S)$. {\it In this section we will write  ${\cal F}$
for the functor }
(\ref{MainFunctor}).
The following result is a particular case of the first part of Theorem \ref{Aus_Buksbaum_3} (those part is proved in Section \ref{One_Lemma}).
\begin{thm}
\label{NisnevichCor}
Let $S$ be an $\mathcal O$-algebra which is discrete valuation ring with fraction field $L$.
Then the map
${\cal F}(S) \to {\cal F}(L)$
is injective.
\end{thm}

\begin{lem}
\label{BoundaryInj}
Let $\mu: \bG \to \bT$ be the above morphism of our reductive group schemes.
Then for an $\mathcal O$-algebra $L$, where $L$ is a field, the boundary map
$\partial: \bT(L)/{\mu (\bG(L))} \to \textrm{H}^1_{\text{\'et}}(L,\bH)$
is injective.
\end{lem}

\begin{proof}
Repeat literally the proof of \cite[Lemma 6.2]{Pan2}.
\end{proof}

Let $k$, $\mathcal O$ and $K$ be as above in this Section.
Let $\mathcal K$ be a field containing $K$ and
$x: \mathcal K^* \to \mathbb Z$
be a discrete valuation vanishing on $K$.
Let $A_x$ be the valuation ring of $x$. Clearly,
$\mathcal O \subset A_x$.
Let
$\hat A_x$ and $\hat {\mathcal K}_x$
be the completions of $A_x$ and $\mathcal K$ with respect to $x$.
Let
$i:\mathcal K \hookrightarrow \hat {\mathcal K}_x$
be the inclusion. By Theorem
\ref{NisnevichCor}
the map
${\cal F}(\hat A_x)\to {\cal F}(\hat{\mathcal K}_x)$
is injective. We will identify
${\cal F}(\hat A_x)$
with its image under this map. Set
$$
{\cal F}_x(\mathcal K)=i_*^{-1}({\cal F}(\hat A_x)).
$$

The inclusion
$A_x\hookrightarrow \mathcal K$
induces a map
$
{\cal F}(A_x) \to {\cal F}(\mathcal K)
$
which is injective by Theorem
\ref{NisnevichCor}.
So both groups
${\cal F}(A_x)$ and ${\cal F}_x(\mathcal K)$
are subgroups of
${\cal F}(\mathcal K)$.
The following lemma shows that
${\cal F}_x(\mathcal K)$
coincides with the subgroup of $
{\cal F}(\mathcal K)$
consisting of all elements {\it unramified} at $x$.

\begin{lem}
\label{TwoUnramified}
${\cal F}(A_x)={\cal F}_x(\mathcal K)$.
\end{lem}

\begin{proof}
Repeat literally the proof of \cite[Lemma 6.3]{Pan2}.
\end{proof}
Let $S$ be an $\mathcal O$-algebra which is an integral domain and suppose $S$ is a regular ring.
Let $L$ be its fraction field.
For each height $1$ prime ideals $\mathfrak p$ in $S$ the group map
${\cal F}(S_{\mathfrak p})\to {\cal F}(L)$ is injective by the first part of Theorem \ref{Aus_Buksbaum_3}.
Define the {\it subgroup of $S$-unramified elements of $\mathcal F (L)$} as
\begin{equation}
\label{DefnUnramified}
\mathcal F_{nr,S}(L)=
\bigcap_{\mathfrak p \in Spec(S)^{(1)}} \mathcal F(S_{\mathfrak p}) \subseteq {\cal F}(L),
\end{equation}
where $Spec(S)^{(1)}$ is the set of height $1$ prime ideals in $S$.
Obviously the image of $\mathcal F(S)$ in $\mathcal F(L)$ is contained in
$\mathcal F_{nr,S}(L)$.
%In some cases
%$\mathcal F(S_{\mathfrak p})$
%injects into
%$\mathcal F(L)$
%and
%$\mathcal F_{nr,S}(L)$
%is simply the intersection of all
%$\mathcal F(S_{\mathfrak p})$.
For
%a regular domain $S$ with the fraction field $\cal K$
%and
each height one prime $\mathfrak p$ in $S$
we construct {\it specialization maps}
$s_{\mathfrak p}: {\cal F}_{nr, S}(L) \to {\cal F} (l(\mathfrak p))$,
where $L$ is the field of fractions of $S$ and
$l(\mathfrak p)$
is the residue field of
$S$ at the prime $\mathfrak p$.

\begin{defn}
\label{SpecializationDef}
Let
$Ev_{\mathfrak p}: \bT(S_{\mathfrak p}) \to \bT(l(\mathfrak p))$
and
$ev_{\mathfrak p}: {\cal F}(S_{\mathfrak p}) \to {\cal F}(K(\mathfrak p))$
be the maps induced by the canonical $S$-algebra homomorphism
$S_{\mathfrak p} \to l(\mathfrak p)$.
Define a homomorphism
$s_{\mathfrak p}: {\cal F}_{nr, S}(L) \to {\cal F} (l(\mathfrak p))$
by
$s_{\mathfrak p}(\alpha)= ev_{\mathfrak p}(\tilde \alpha)$,
where
$\tilde \alpha$
is a lift of $\alpha$ to
${\cal F}(S_{\mathfrak p})$.
Theorem
\ref{NisnevichCor}
shows that the map $s_{\mathfrak p}$ is well-defined.
It is called the specialization map. The map $ev_{\mathfrak p}$ is called the evaluation
map at the prime $\mathfrak p$.

Obviously for
$\alpha \in \bT(S_\mathfrak p)$
one has
$s_{\mathfrak p}(\bar \alpha)=\overline {Ev_{\mathfrak p}(\alpha)}
\in {\cal F}(l(\mathfrak p))$.
\end{defn}

Let $k$, $\mathcal O$ and $K$ be as above in this Section.
The following two results are proved using literally the same arguments as in the proof of \cite[Thm. 6.5]{Pan2} and \cite[Cor. 6.6]{Pan2} respectively.
%We need the following theorem and its corollary.
\begin{thm}[Homotopy invariance]
\label{HomInvNonram}
Let $K(t)$ be the rational function field in one variable over the field $K$.
%Let $S \mapsto {\cal F}(S)$ be the functor defined by the formulae
%(\ref{MainFunctor2})
Define ${\cal F}_{nr,K[t]}(K(t))$ by the formulae
(\ref{DefnUnramified}).
Then one has an equality
$$
{\cal F}(K)={\cal F}_{nr,K[t]}(K(t)).
$$
\end{thm}

\begin{cor}
\label{TwoSpecializations}
%Let $S \mapsto {\cal F}(S)$ be the functor defined in
%(\ref{MainFunctor}).
Let
$$
s_0, s_1: {\cal F}_{nr, K[t]}(K(t)) \to {\cal F}(K)
$$
be the specialization maps at zero and at one
(at the primes (t) and (t-1)). Then $s_0=s_1$.
\end{cor}

%The following result is \cite[Lemma 4.0.11]{Pa2}
%up to change of notation.
\begin{lem}
\label{KeyUnramifiedness}
Let $B \subset A$ be a finite extension of $K$-smooth algebras, which are integral domains
and each has dimension one.
Let $0 \neq f \in A$ and
let $h \in B\cap fA$ be such that the induced map
$B/hB\to A/fA$ is an isomorphism.
Suppose
$hA=fA\cap J^{\prime\prime}$
for an ideal
$J^{\prime\prime} \subseteq A$
co-prime to the principal ideal $fA$.

%Suppose $N_{B/A}(f)=fg \in B$ for a certain $g \in B$ coprime
%with $f$. Suppose the composite map
%$A/N(f)A \to B/N(f)B \to B/fB$
%is an isomorphism.
Let $E$ and $F$ be the field of fractions of $B$ and $A$ respectively.
Let $\alpha \in \bT(A_f)$ be such that
$\bar \alpha \in {\cal F}(F)$
is $A$-unramified. Then, for
$\beta= N_{F/E}(\alpha)$,
the class
$\bar \beta \in {\cal F}(E)$
is $B$-unramified.
\end{lem}

\begin{proof}
Repeat literally the proof of \cite[Lemma 6.7]{Pan2}.
\end{proof}

\section{Few recollections}
\label{First Application}
Let $X$ be an affine $k$-smooth irreducible $k$-variety, and let $x_1,x_2,\dots,x_n$ be closed points in $X$.
Let $\mathcal O$ be the semi-local ring $\mathcal O_{X,\{x_1,x_2,\dots,x_n\}}$.
Let $U=Spec(\mathcal O)$ and $can: U\hra X$ be the canonical embedding.
Let $\bG$ be a reductive
$X$-group scheme
and let
$\bG_U= can^*(\bG)$
be the pull-back of $\bG$ to $U$.
Let $\bT$ be an $X$-torus
and let
$\bT_U= can^*(\bT)$
be the pull-back of $\bG$ to $U$.
Let $\mu: \bG \to \bT$
be an $X$-group scheme morphism which is smooth as an $X$-scheme morphism.
Let $\mu_U=can^*(\mu)$.
The following result is \cite[Theorem 4.1]{Pan2}.
%\cite[Theorem 6.1]{Pan1}.

\begin{thm}
\label{equating3}
Given a non-zero function $\emph{f}\in k[X]$ vanishing at each point $x_i$,
there is a diagram of the form
\begin{equation}
\label{DeformationDiagram0}
    \xymatrix{
\Aff^1 \times U\ar[drr]_{\pr_U}&&\mathcal X^{\prime} \ar[d]^{}
\ar[ll]_{\sigma}\ar[d]_{q_U}
\ar[rr]^{q_X}&&X &\\
&&U \ar[urr]_{\can}\ar@/_0.8pc/[u]_{\Delta^{\prime}} &\\
    }
\end{equation}
with an irreducible affine scheme $\mathcal X^{\prime}$, a smooth morphism $q_U$, a finite surjective $U$-morphism $\sigma$ and an essentially smooth morphism $q_X$,
and a function $f^{\prime} \in q^*_X(\emph{f} \ )k[\mathcal X^{\prime}]$,
which enjoys the following properties:
\begin{itemize}
\item[\rm{(a)}]
if
$\mathcal Z^{\prime}$ is the closed subscheme of $\mathcal X^{\prime}$ defined by the ideal
$(f^{\prime})$, then the morphism
$\sigma|_{\mathcal Z^{\prime}}: \mathcal Z^{\prime} \to \Aff^1\times U$
is a closed embedding and the morphism
$q_U|_{\mathcal Z^{\prime}}: \mathcal Z^{\prime} \to U$ is finite;
\item[\rm{(a')}] $q_U\circ \Delta^{\prime}=id_U$ and $q_X\circ \Delta^{\prime}=can$ and $\sigma\circ \Delta^{\prime}=i_0$, \\
where $i_0$ is the zero section of the projection $\pr_U$;
\item[\rm{(b)}] $\sigma$
is \'{e}tale in a neighborhood of
$\mathcal Z^{\prime}\cup \Delta^{\prime}(U)$;
\item[\rm{(c)}]
$\sigma^{-1}(\sigma(\mathcal Z^{\prime}))=\mathcal Z^{\prime}\coprod \mathcal Z^{\prime\prime}$
scheme theoretically
for some closed subscheme $\mathcal Z^{\prime\prime}$
\\ and
$\mathcal Z^{\prime\prime} \cap \Delta^{\prime}(U)=\emptyset$;
\item[\rm{(d)}]
$\mathcal D_0:=\sigma^{-1}(\{0\} \times U)=\Delta^{\prime}(U)\coprod \mathcal D^{\prime}_0$
scheme theoretically
for some closed subscheme $\mathcal D^{\prime}_0$
and $\mathcal D^{\prime}_0 \cap \mathcal Z^{\prime}=\emptyset$;
\item[\rm{(e)}]
for $\mathcal D_1:=\sigma^{-1}(\{1\} \times U)$ one has
%$D^{\prime}_1 \cap \mathcal Z^{\prime}=\emptyset$.
$\mathcal D_1 \cap \mathcal Z^{\prime}=\emptyset$.
\item[\rm{(f)}]
there is a monic polynomial
$h \in \mathcal O[t]$
with
$(h)=Ker[\mathcal O[t] \xrightarrow{\sigma^*} k[\mathcal X^{\prime}] \xrightarrow{-} k[\mathcal X^{\prime}]/(f^{\prime})]$, \\
where the map bar takes any $g\in k[\mathcal X^{\prime}]$ to ${\bar g}\in k[\mathcal X^{\prime}]/(f^{\prime})$;\\
\item[\rm{(g)}] there are $\mathcal X^{\prime}$-group scheme isomorphisms
$\Phi:  q^*_U(\bG_U)\to q^*_X(\bG)$,
$\Psi:  q^*_U(\bT_U)\to q^*_X(\bT)$
with
$(\Delta^{\prime})^*(\Phi)= id_{\bG_U}$,
$(\Delta^{\prime})^*(\Psi)= id_{\bT_U}$
and
$q^*_X(\mu) \circ \Phi=\Psi \circ q^*_U(\mu_U)$.
\end{itemize}
\end{thm}

\begin{rem}
The triple $(q_U: \mathcal X^{\prime} \to U, f^{\prime}, \Delta^{\prime})$ is
{\it a nice triple} over $U$, since $\sigma$ is
a finite surjective $U$-morphism.
See \cite[Defn.3.1]{PSV} for the definition of a nice triple.

The morphism
$q_X$ is not equal to $can\circ q_U$, since
$f^{\prime} \in q^*_X(\textrm{f})k[\mathcal X^{\prime}]$
and the morphism
$q_U|_{\mathcal Z^{\prime}}: \mathcal Z^{\prime}=\{f^{\prime}=0\} \to U$ is finite.

%We stress that the scheme $\mathcal X^{\prime}$ from Theorem \ref{DeformationDiagram0}
%does not coincide usually with the scheme $\mathcal X^{\prime}$ as in
%\cite[Theorem 6.1]{Pan1}.
\end{rem}
To formulate a consequence of the theorem
\ref{equating3} (see Corollary \ref{ElementaryNisSquareNew_1}),
note that using the items (b) and (c) of Theorem
\ref{equating3}
one can find an element
$g \in I(\mathcal Z^{\prime\prime})$
such that \\
(1) $(f^{\prime})+(g)=\Gamma(\mathcal X^{\prime}, \mathcal O_{\mathcal X^{\prime}})$, \\
(2) $Ker((\Delta^{\prime})^*)+(g)=\Gamma(\mathcal X^{\prime}, \mathcal O_{\mathcal X^{\prime}})$, \\
(3) $\sigma_g=\sigma|_{\mathcal X^{\prime}_g}: \mathcal X^{\prime}_g \to \Aff^1_U$ is \'{e}tale.\\

Here is the corollary. It is proved in \cite[Cor. 7.2]{Pan0}.
\begin{cor}
\label{ElementaryNisSquareNew_1}
The function $f^{\prime}$ from Theorem \ref{equating3}, the polynomial $h$ from the item $(\textrm{f} \ )$
of that Theorem, the morphism $\sigma: \mathcal X^{\prime} \to \Aff^1_U$
and the function
$g \in \Gamma(\mathcal X,\mathcal O_{\mathcal X} )$
defined just above
enjoy the following properties:
\begin{itemize}
\item[\rm{(i)}]
the morphism
$\sigma_g= \sigma|_{\mathcal X^{\prime}_g}: \mathcal X^{\prime}_g \to \Aff^1\times U $
is \'{e}tale,
\item[\rm{(ii)}]
data
$ (\mathcal O[t],\sigma^*_g: \mathcal O[t] \to \Gamma(\mathcal X^{\prime},\mathcal O_{\mathcal X^{\prime}})_g, h ) $
satisfies the hypotheses of
\cite[Prop.2.6]{C-TO},
i.e.
$\Gamma(\mathcal X^{\prime},\mathcal O_{\mathcal X^{\prime}} )_g$
is a finitely generated
$\mathcal O[t]$-algebra, the element $(\sigma_g)^*(h)$
is not a zero-divisor in
$\Gamma(\mathcal X^{\prime},\mathcal O_{\mathcal X^{\prime}} )_g$
and
$\mathcal O[t]/(h)=\Gamma(\mathcal X^{\prime},\mathcal O_{\mathcal X^{\prime}})_g/h\Gamma(\mathcal X^{\prime},\mathcal O_{\mathcal X^{\prime}})_g$ \ ,
\item[\rm{(iii)}]
$(\Delta(U) \cup \mathcal Z') \subset \mathcal X^{\prime}_g$ \ and $\sigma_g \circ \Delta=i_0: U\to \Aff^1\times U$,
\item[\rm{(iv)}]
$\mathcal X^{\prime}_{gh} \subseteq \mathcal X^{\prime}_{gf^{\prime}}\subseteq \mathcal X^{\prime}_{f^{\prime}}\subseteq \mathcal X^{\prime}_{q^*_X(\emph{f})}$ \ ,
\item[\rm{(v)}]
$\mathcal O[t]/(h)=\Gamma(\mathcal X^{\prime},\mathcal O_{\mathcal X^{\prime}})/(f^{\prime})$,
$h\Gamma(\mathcal X^{\prime},\mathcal O_{\mathcal X^{\prime}})=(f^{\prime})\cap I(\mathcal Z^{\prime\prime})$
and
$(f^{\prime}) +I(\mathcal Z^{\prime\prime})=\Gamma(\mathcal X^{\prime},\mathcal O_{\mathcal X^{\prime}})$.
\end{itemize}
\end{cor}

\section{Purity}\label{purity}
Let $S$ be a regular ring, $\bG$ a reductive $S$-group scheme, $\bT$ an $S$-torus,
$\mu: \bG\to \bT$ an $S$-group scheme morphisms which is smooth as a scheme morphism.
Suppose $S$ is an integral domain. Let $L$ be its field of fractions. For each $S$-algebra $S'$
write ${\cal F}(S')$ for the group $\bT(S')/\mu(\bG(S'))$. For any $a\in \bT(S')$ write $\bar a$
for the class of $a$ in ${\cal F}(S')$. Let $\mathfrak p$ be a height one prime ideal in $S$,
then by Theorem \ref{Aus_Buksbaum_3} the group ${\cal F}(S_{\mathfrak p})$ is a subgroup in ${\cal F}(L)$.

Recall some notion. For an element $a\in \bT(L)$ and a height one prime $\mathfrak p\subset S$
we say that $\bar a\in {\cal F}(L)$ is {\it unramified at} $\mathfrak p$,
if $\bar a$ is in ${\cal F}(S_{\mathfrak p})$.
We say that the element $\bar a\in {\cal F}(L)$ is $S$-unramified if for any height one prime ideal $\mathfrak p$ in $S$
the element $\bar a$ is in ${\cal F}(S_{\mathfrak p})$. Clearly, the image of ${\cal F}(S)$ in ${\cal F}(L)$
is in $\cap \ {\cal F}(S_{\mathfrak p})$, where the intersection is taken over all height one primes of $S$.
We say that {\it purity holds for the ring} $S$ if
$$Im[{\cal F}(S)\to {\cal F}(L)]=\cap \ {\cal F}(S_{\mathfrak p}).$$
Equivalently, purity holds for $S$ if each $S$-unramified element of ${\cal F}(L)$ comes from ${\cal F}(S)$.
Clearly, the sequence
$\{1\} \to {\cal F}(S_{\mathfrak p}) \to
{\cal F}(L) \xrightarrow{r_{\mathfrak p}}
\bT(L)/[\bT(R_{\mathfrak p})\cdot \mu(\bG(L))] \to \{1\}
$
is exact, where $r_{\mathfrak p}$ is the factorization map.
Thus, an element $a\in \bT(L)$ its class $\bar a$ in ${\cal F}(L)$ is unramified at $\mathfrak p$,
if and only if $r_{\mathfrak p}(\bar a)=0$. Hence
purity holds for $S$ if and only if the sequence
${\cal F}(S)\to {\cal F}(L)\xrightarrow{\sum r_{\mathfrak p}} \oplus_{\mathfrak p} \bT(L)/[\bT(R_{\mathfrak p})\cdot \mu(\bG(L))]$
is exact. Our aim is to prove the following assertion:\\
%Let $M\subset S$ be a multiplicative set and $S_M$ the localization of $S$.
$(\ast)$ Purity holds for the ring $R$, group schemes $\bG$,$\bT$ and the morphism $\mu$ as in Theorem \ref{Aus_Buksbaum_3}.\\\\
The proof is subdivided in few steps.\\
{\it Claim 1.}
Let $X$ be a $k$-smooth irreducible affine $k$-variety.
Let $\bG$ be a reductive $X$-group scheme, $\bT$ be an $X$-torus and
$\mu: \bG \to \bT$ be an $X$-group scheme morphism which is smooth as an $X$-scheme morphism.
Suppose the $k$-algebra $R$ is the semi-local ring of finitely many closed points
on  $X$. Then purity holds for $R$.

To prove this Claim we just repeat literally the proof of
\cite[Theorem 1.1]{Pan2} replacing references to
\cite[Corollary 4.3, (ii),(v)]{Pan2}
with the one to the items (ii) and (v) of
Corollary \ref{ElementaryNisSquareNew_1}.
Replacing also references to \cite[Lemma 6.7]{Pan2}
with the one to
Lemma \ref{KeyUnramifiedness}.
Replacing also references to \cite[Theorem 6.5]{Pan2}
with the one to Theorem \ref{HomInvNonram}.
Replacing also references to \cite[Corollary 6.6]{Pan2}
with the one to Corollary \ref{TwoSpecializations}.
Replacing also references to \cite[Theorem 4.1]{Pan2}
with the one to Theorem \ref{equating3}.
Replacing also references to \cite[Definition 6.4]{Pan2}
with the one to the remark at the end of Definition \ref{SpecializationDef}.
The Claim 1 is proved.\\\\
{\it Claim 2.}
Let $X$ be a $k$-smooth irreducible affine $k$-variety and $\xi_1,...,\xi_n$ be points of the scheme $Spec(k[X])$
such that for each pair $r,s$ the point $\xi_r$ is not in the closure $\overline {\{\xi_s\}}$ of $\xi_s$.
Let $R$ be the semi-local ring $\cO_{X,\xi_1,...,\xi_n}$ of scheme points $\xi_1,...,\xi_n$ of $Spec(k[X])$.
Let $\bG$ be a reductive $X$-group scheme, $\bT$ be an $X$-torus and
$\mu: \bG \to \bT$ be an $X$-group scheme morphism which is smooth as an $X$-scheme morphism.
Then purity holds for $R$.

To prove this Claim take an element $a \in \bT(k(X))$ such that $\bar a$ is unramified at each irreducible divisor $D$ containing
at least one of the points $\xi_r$. We have to prove that the element $\bar a\in {\cal F}(K)$ is in the image of ${\cal F}(R)$.
Clearly, there is a non-zero $f\in k[X]$ such that $a \in \bT(k[X_f])$. Write down the divisor $div(f)\in Div(X)$
in the form $div(f)=\Sigma m_iD_i + \Sigma n_jD'_j$ such that for each index $i$ there is an index $r$ with $\xi_r\in D_i$
and for any index $j$ and any index $r$ the point $\xi_r$ does not belong to $D'_j$.
There is an element $g\in k[X]$
such that for any index $j$ one has $D'_j$ is contained in the closed subset $\{g=0\}$ and $g$ does not belong to any of $\xi_r$'r.
Replacing $X$ with $X_g$ we see that $a \in \bT(k[X_f])$, $div(f)=\Sigma m_iD_i$ and $\bar a$ is unramified at each irreducible divisor $D_i$.
Hence $\bar a$ is unramified at each height one prime ideal of $k[X]$.
Our assumption on points $\xi_r$'s yield the following: one can choose
closed points $x_r\in \overline {\{\xi_s\}}$ such that for each $r\neq s$ the point $x_r$ is not in $\overline {\{\xi_s\}}$.
Particularly, for each $r\neq s$ one has $x_r\neq x_s$. The element $\bar a$ is unramified at each height one prime ideal of $k[X]$.
Thus, by Claim 1 the element $\bar a$ is in the image of ${\cal F}(\cO_{X,x_1,...,x_n})$.
So, the element $\bar a$ is in the image of ${\cal F}(\cO_{X,\xi_1,...,\xi_n})={\cal F}(R)$.
The Claim 2 is proved.\\\\
{\it Claim 3.}
The assertion $(\ast)$ is true.\\
In the rest of the section we prove Claim 3.
Clearly, we may assume that $k$ is a prime field and hence $k$ is perfect.
It follows from Popescu's theorem [Pop, Swa, Spi] that $R$
is a filtered inductive limit of smooth $k$-algebras $R_{\alpha}$. Modifying the inductive system
$R_{\alpha}$
if necessary, we can assume that each $R_{\alpha}$ is integral.
For each maximal ideal $\mathfrak m_i$ in $R$ ($i = 1,...,n$) set
$\mathfrak p_i = \phi^{-1}_{\alpha}(\mathfrak m_i)$.
The homomorphism $\phi_{\alpha}: R_{\alpha}\to R$ induces a homomorphism of semi-local rings
$\phi^{\prime}_{\alpha}: (R_{\alpha})_{\mathfrak p_1,...,\mathfrak p_n} \to R$.
Since this moment we will write $A_{\alpha}$ for $(R_{\alpha})_{\mathfrak p_1,...,\mathfrak p_n}$
and $A$ for $R$ (to keep consistency of notation).
%and write $\phi_{\alpha}$ for $\phi^{\prime}_{\alpha}$. Finally, we will write
Thus, $A$
is a filtered inductive limit of regular semi-local $k$-algebras $A_{\alpha}$.

There exist an index $\alpha$, a reductive
group scheme $\bG_{\alpha}$, a torus $\bT_{\alpha}$ over $A_{\alpha}$ and
an $A_{\alpha}$-group scheme morphism $\mu_{\alpha}: \bG_{\alpha} \to \bT_{\alpha}$
which is smooth as an $A_{\alpha}$-scheme morphism
such that
$\bG=\bG_{\alpha}\times_{Spec(A_{\alpha})} Spec(A)$,
$\bT=\bT_{\alpha}\times_{Spec(A_{\alpha})} Spec(A)$,
$\mu=\mu_{\alpha}\times_{Spec(A_{\alpha})} Spec(A)$.
Replacing the index system with a co-final one consisting of indexes $\beta\geq \alpha$,
we may and will suppose that the reductive
group scheme
$\bG$, the torus $\bT$ and the group scheme morphism $\mu: \bG\to \bT$
comes from $A_{\alpha}$, and $\mu$ is is smooth as an $A_{\alpha}$-scheme morphism.
These observations and Claim 2 yield the following intermediate result \\
$(\ast\ast)$ for these
$\bG$, $\bT$ and $\mu: \bG\to \bT$ over $A_{\alpha}$ purity holds
for each ring $A_{\beta}$ with $\beta\geq \alpha$.

Let now $K$ be the field of fractions of $A$ and, for each
$\beta\geq \alpha$,  let $K_\beta$ be the field of fractions of $A_\beta$.
For each index $\beta\geq \alpha$ let
$\mathfrak a_{\beta}$
be the kernel of the map
$\phi^{\prime}_{\beta}: A_{\beta} \to A$
and
$B_{\beta}=(A_{\beta})_{\mathfrak a_{\beta}}$.
Clearly, for each
$\beta\geq \alpha$,  $K_{\beta}$
is the field of fractions of
$B_{\beta}$.
The composition map
$A_{\beta} \to A \to K$
factors through $B_{\beta}$.
%and hence it also factors through the residue field $k_{\beta}$
%of the local ring $B_{\beta}$.
Since $A$ is a filtering direct limit of the $A_{\beta}$'s
we see that $K$ is a filtering direct limit of the $B_{\beta}$'s.
We will write
$\psi_{\beta}$
for the canonical morphism
$B_{\beta} \to K$.

\begin{lem}
\label{WeakGrothendieck}
For each index $\alpha$ the group map $W(B_{\alpha})\to W(K_{\alpha})$ is injective.
\end{lem}

\begin{proof}
Just apply the first part of Theorem \ref{Aus_Buksbaum_3} to the $k$-algebra $B_{\alpha}$.
\end{proof}

\begin{lem}
\label{LiftToFiniteLevel}
Let $a\in W(K)$ be an $A$-unramified element.
Then there exists an index $\alpha$ and an element
$b_{\alpha} \in W(B_{\beta})$
such that
$\psi_{\alpha}(b_{\alpha})=b$
and the class
$b_{\alpha} \in W(K_{\beta})$
is $A_{\alpha}$-unramified.
\end{lem}

\begin{proof}
Repeat literally the proof of \cite[Lemma 9.0.9]{P1}. It works
for the semi-local case as well.
\end{proof}
We complete the proof of Claim 3 as follows.
Let $a\in W(K)$ be an $A$-unramified element.
We have to check that it comes from $W(A)$.
By Lemma \ref{LiftToFiniteLevel}
there exists an index $\alpha$ and an element
$b_{\alpha} \in W(B_{\alpha})$
such that
$\psi_{\alpha}(b_{\alpha})=b$
and the class
$b_{\alpha} \in W(K_{\beta})$
is $A_{\alpha}$-unramified.
For this index $\alpha$ consider a commutative diagram of $k$-algebras
$$
\xymatrix{
    A_{\alpha}  \ar[d]\ar[rr]^-{\varphi_{\alpha}}\ar[d]_-{} && A \ar[d]^-{} &   \\
     B_{\alpha} \ar[d]\ar[rr]^-{\psi_{\alpha}}  &&  K \\
     K_{\alpha}.
    }
$$
The class
$\bar b_{\beta} \in {\cal F}(K_{\beta})$
is $A_{\beta}$-unramified.
Hence by the statement $(\ast\ast)$ there exists an element
$a_{\beta} \in \bT(A_{\beta})$
such that
$\bar b_{\beta}=\bar a_{\beta}$ in ${\cal F}(K_{\beta})$.
By Lemma
\ref{WeakGrothendieck}
one has an equality
$\bar b_{\beta}=\bar a_{\beta}$ in
${\cal F}(B_{\beta})$.
Hence $\bar b \in {\cal F}(K)$ coincides with
the image of the element $\phi_{\beta} (\bar a_{\beta})$
in ${\cal F}(K)$.
The Claim 3 is proved.
Thus, the sequence
\eqref{Aus_Buks_sequence_3} is exact at its middle term.

\section{Proof of Theorem \ref{Aus_Buksbaum_3}}
\label{SectAus_Buksbaum_3}
\begin{proof}[Proof of Theorem \ref{Aus_Buksbaum_3}]
The proof of the first assertion of Theorem \ref{Aus_Buksbaum_3} is given in Section \ref{One_Lemma}.
The exactness of the sequence
\eqref{Aus_Buks_sequence_3} at its middle term is proved in Section \ref{purity}.

Prove now the surjectivity of the map $\sum r_{\mathfrak p}$.
Clearly, it is sufficient to prove
the surjectivity of the map
$\mathbf T(K) \xrightarrow{\sum r'_{\mathfrak p}} \bigoplus_{\mathfrak p} \mathbf T(K)/\mathbf T(R_{\mathfrak p})$,
where $\mathfrak p$ runs over
the height $1$ primes of $R$ and $r'_{\mathfrak p}$ is the factorisation map.
We follow arguments from
\cite[Section 9]{Pan3}.

We prefer to switch to the scheme terminology.
Set $X:=Spec(R)$.
Consider a finite \'{e}tale Galois morphism $\pi: \tilde X\to X$ with irreducible $\tilde X$
such that the torus $\bT$ splits over $\tilde X$. Let $Gal:=Aut(\tilde X/X)$ be its Galois group.
Since the torus $\bT$ splits over $\tilde X$ we have a short exact sequence of $Gal$-modules
$$0\to \bT(\tilde \cO)\to \bT(\tilde K)\to \oplus_{y} \bT(\tilde K)/\bT(\tilde \cO_{X,y})\to 0,$$
where $\tilde \cO=\Gamma(\tilde X, \mathcal O_{\tilde X})$, $\tilde K$ is the fraction field of $\tilde \cO$,
%$\tilde K=k(\tilde X)$, $\tilde \cO$ is the semi-local ring $\cO_{\tilde X, \tilde x}$ of the finite set $\tilde x=\pi^{-1}(x)$,
$y$ runs over the set $X^{(1)}$ of all
codimension $1$ points of $X$
%{\bf the local scheme $U:=Spec(\mathcal O)$
%}
and for any $y\in X^{(1)}$ the ring
$\tilde \cO_{X,y}$ is the semi-local ring $\cO_{\tilde X, \tilde y}$ of the finite set $\tilde y=\pi^{-1}(y)$
on the scheme $\tilde X$. Write $\cO$ for $R$ to be consistent with the above notation.

The above short exact sequence of $Gal$-modules gives rise to a long exact sequence of $Gal$-cohomology groups of the form
$$0\to \bT(\cO)\xrightarrow{in} \bT(K)\to \oplus_{y} [\bT(\tilde K)/\bT(\tilde \cO_{X,y})]^{Gal} \to H^1(Gal,\bT(\tilde \cO))\xrightarrow{H^1(in)} H^1(Gal,\bT(\tilde K)).$$
We claim that the map $H^1(in)$ is a monomorphism. Indeed, the group $H^1(Gal,\bT(\tilde \cO))$ is a subgroup of the group $H^1_{et}(X, \bT)$
and the group $H^1(Gal,\bT(\tilde K))$ is a subgroup of the group $H^1_{et}(Spec~K, \bT_K)$. By Theorem \ref{MainThm1}
the group map $H^1_{et}(X, \bT)\to H^1_{et}(Spec~K, \bT_K)$ is a monomorphism. Thus, $H^1(in)$ is a monomorphism also.
So, we have a short exact sequence of the form
$0\to \bT(\cO)\xrightarrow{in} \bT(K)\to \oplus_{y} [\bT(\tilde K)/\bT(\tilde \cO_{X,y})]^{Gal}\to 0$.

There is also the complex $0\to \bT(\cO)\xrightarrow{in} \bT(K)\to \oplus_{y} \bT(K)/\bT(\cO_{X,y})$. Set
$\alpha=id_{\bT(\cO)}$, $\beta=id_{\bT(K)}$ and let $\gamma=\oplus_{y} \gamma_y $, where
$\gamma_y: \bT(K)/\bT(\cO_{X,y})\to [\bT(\tilde K)/\bT(\tilde \cO_{X,y})]^{Gal}$ is induced by the inclusion $K\subset \tilde K$.
The maps $\alpha$, $\beta$ and $\gamma$ form a morphism between this complex and the above short exact sequence. We claim that this morphism is an isomorphism.
This claim completes the proof of the theorem.
%We will prove this claim in the case
%when the torus $C$ is defined over $k$. The general case requires small extra arguments
%and we leave it to the reader.

To prove this claim it is sufficient to prove that $\gamma$ is an isomophism. Since the map
$\bT(K)\to \oplus_{y} [\bT(\tilde K)/\bT(\tilde \cO_{X,y})]^{Gal}$
is an epimorphism, hence so is the map $\gamma$.
It remains to prove that $\gamma$ is a monomorphism.
To do this it is sufficient to check that for any point $y\in X^{(1)}$ the map
$\bT(K)/\bT(\cO_{X,y})\to \bT(\tilde K)/\bT(\tilde \cO_{X,y})$
is a monomorphism. We will write $\epsilon_y$ for the latter map.
We prove below that $ker(\epsilon_y)$ is a torsion group and
the group $\bT(K)/\bT(\cO_{X,y})$ has no torsion. These two claims show that the map
$\epsilon_y$ is injective indeed.

To prove that $ker(\epsilon_y)$ is a torsion group recall that there are norm maps
$N_{\tilde \cO_{X,y}/\cO_{X,y}}: \bT(\tilde \cO_{X,y})\to \bT(\cO_{X,y})$ and
$N_{\tilde K/K}: \bT(\tilde K)\to \bT(K)$ (see Section \ref{SectNorms}). These maps induce a homomorphism
$$N_y: \bT(\tilde K)/\bT(\tilde \cO_{X,y})\to \bT(K)/\bT(\cO_{X,y})$$
such that $N_y\circ \epsilon_y$ is the multiplication by the degree $d$ of $\tilde K$ over $K$.
Thus, $ker(\epsilon_y)$ is killed by the multiplication by $d$.

Show now that the group $\bT(K)/\bT(\cO_{X,y})$ has no torsion. Take an element $a_K\in \bT(K)$ and suppose that
its class in $\bT(K)/\bT(\cO_{X,y})$ is a torsion element. Let $\tilde a_K$ be the image of $a_K$ in $\bT(\tilde K)$.
Since $\bT$ splits over $\tilde K$ we see that $\bT(\tilde K)/\bT(\tilde \cO_{X,y})$ is torsion free. Thus,
the class of $\tilde a_K$ in $\bT(\tilde K)/\bT(\tilde \cO_{X,y})$ vanishes. So, there is a unique element $\tilde a$ in $\bT(\tilde \cO_{X,y})$
whose image in $\bT(\tilde K)$ is $\tilde a_K$.
Moreover, $\tilde a$ is a $Gal$-invariant element in $\bT(\tilde \cO_{X,y})$, because $\tilde a_K$ comes from $\bT(K)$.
%Hence $\tilde a$ is in
Since $\bT(\tilde \cO_{X,y})^{Gal}=\bT(\cO_{X,y})$, there is a unique element $a\in \bT(\cO_{X,y})$
whose image in $\bT(\tilde \cO_{X,y})$ is $\tilde a$. Clearly, the image of $a$ into $\bT(K)$
is the element $a_K$. Thus, the class of $a_K$ in $\bT(K)/\bT(\cO_{X,y})$ vanishes.
So, the group $\bT(K)/\bT(\cO_{X,y})$ is torsion free.
The injectivity of $\epsilon_y$ is proved.
The surjectivity of the map $\sum r_{\mathfrak p}$ is proved.
Theorem \ref{Aus_Buksbaum_3} is proved.
\end{proof}

\end{document}